\documentclass[preprint,12pt]{elsarticle}

\usepackage{amssymb}
\usepackage{url}
\usepackage{amsthm}
\usepackage{amsmath}

\newtheorem{lem}{Lemma}

\newtheorem{thm}{Theorem}
\newtheorem{prp}{Property}
\newtheorem{pro}{Proposition}
\newtheorem{cor}{Corollary}

\newtheorem{exm}{Example}

\journal{}

\begin{document}

\begin{frontmatter}

\title{Some new general lower bounds for mixed metric dimension of graphs}

\author{Milica Milivojevi\'c Danas\fnref{kg}}
\ead{milica.milivojevic@kg.ac.rs}

\author{Jozef Kratica \fnref{mi}}
\ead{jkratica@gmail.com}

\author{Aleksandar Savi\'{c}  \fnref{matf}}
\ead{aleks3rd@gmail.com}

\author[va]{Zoran Lj. Maksimovi\'{c}\corref{cor1}}
\ead{zoran.maksimovic@gmail.com}

\address[kg]{Faculty of Science and Mathematics, University of Kragujevac, Radoja Domanovi\' ca 12, Kragujevac, Serbia }
\address[mi]{Mathematical Institute, Serbian Academy of Sciences and Arts, Kneza Mihaila 36/III, 11 000 Belgrade, Serbia}
\address[matf]{Faculty of Mathematics, University of Belgrade, Studentski trg 16/IV, 11000 Belgrade, Serbia}
\address[va]{University of Defence, Military Academy, Generala Pavla Juri\v{s}i\'{c}a \v{S}turma 33, 11000 Belgrade, Serbia}

\begin{abstract}
Let $G=(V,E)$ be a connected simple graph. The distance $d(u,v)$ between vertices $u$ and $v$ from $V$ is the number of edges in the shortest $u-v$ path. If $e=uv \in E$ is an edge in $G$ than distance $d(w,e)$ where $w$ is some vertex in $G$ is defined as $d(w,e)=\min(d(w,u),d(w,v))$. Now we can say that vertex $w \in V$ resolves two elements $x,y \in V \cup E$ if $d(w,x) \neq d(w,y)$. The mixed resolving set is a set of vertices $S$, $S\subseteq V$ if and only if any two elements of $E \cup V$ are resolved by some element of $S$. A minimal resolving set related to inclusion is called mixed resolving basis, and its cardinality is called the mixed metric dimension of a graph $G$.

This graph invariant is recently introduced and it is of interest to find its general properties and determine its values for various classes of graphs. Since the problem of finding mixed metric dimension is a minimization problem, of interest is also to find lower bounds of good quality. This paper will introduce three new general lower bounds. The exact values of mixed metric dimension for torus graph is determined using one of these lower bounds. Finally, the comparison between new lower bounds and those known in the literature will be presented on two groups of instances:
\begin{itemize}
\item all 21 conected graphs of order 5;
\item selected 12 well-known graphs with order from 10 up to 36.
\end{itemize}
\end{abstract}


\end{frontmatter}


\section{Introduction}

\subsection{Literature review}

Problem of metric dimension was introduced in the seventies by \cite{sla75} and \cite{har76} independently. Let $G=(V,E)$ be a connected simple graph. Distance in the graph will be measured by number of edges in the shortest $u-v$ path in $G$ where $u,v \in V$, and denoted as $d_G(u,v)$. Then, for vertex $w$ will be said that \textit{resolves} vertices $u$ and $v$ if $d_G(w,u) \neq d_G(w,v)$. Set $S \subseteq V$ will be denoted \textit{resolving set} if any pair of vertices from $V$ are resolved by some element from $S$. If we order set $S$ as $S=\{w_1,\ldots,w_k\}$ than for every vertex $u \in V$ we can determine vector of resolving coordinates given by $r_S(u)=(d_G(u,w_1),\ldots,d_G(u,w_k))$. In this context, $S$ is resolving set if every vertex $u$ has unique vector of resolving coordinates. Minimal resolving set, in the sense of inclusion, is called \textit{metric basis} and its cardinality \textit{metric dimension} of graph $G$. Metric dimension of graph $G$ will be denoted as $\beta(G)$. The term resolving set was introduced by Harrary and Melter, while Slater used term \textit{locating sets}. In the literature synonymous to these terms is also \textit{metric generators}. In order to simplify notation we will replace $d_G(u,v)$ with $d(u,v)$ and denote metric dimension as $\beta(G)$ for graph $G$. Also, in the rest of this article we will assume that graphs in considerations are always connected and simple.

There are several proposed applications for metric dimension in the literature. Originally, Slater considered unique recognition of intruders in the network, while others observed problems of navigating robots in networks \cite{khul96}, chemistry \cite{john93,char00}, some applications in pattern recognition and image processing \cite{mel84}.

After initial works, some variations of this problem were introduced such as resolving dominating sets \cite{bri03}, independent resolving sets \cite{char03}, strong metric dimension \cite{seb04}, local metric dimension \cite{oka10} among others.

As metric dimension and resolving sets give some information about vertices of the graph, it is natural to ask if there is some parameter, or graph invariant, which deals in the same way with graph's edges. Answer to that question was given by Kelenc et al. in \cite{kel18}, where authors introduced edge metric dimension of graphs. The distance between vertex and edge in a graph was defined as smaller of distances between given vertex and endpoints of a given edge. Formally, if $w \in V$ and $e=\{u,v\}$ then

\begin{equation}\label{dis}
d(w,e)=\min(d(w,u),d(w,v)).
\end{equation}

Now, vertex $w$ resolves two edges $e_1$ and $e_2$ if $d(w,e_1) \neq d(w,e_2)$. Similarly as for metric dimension, the edge resolving set $S \subseteq  V$ is a set of vertices such that for any pair of edges from $E$, there is some element in $S$ that resolves them. Consequently, the minimal edge resolving set is edge metric basis and its cardinality is called edge metric dimension of graph and is denoted as $\beta_E(G)$.

Finally, as there is metric and edge metric dimension of a graph, Kelenc et al. (2017) in \cite{yer17} recently introduced a concept of mixed metric dimension and initiated the study of its mathematical properties. In their article, vertex $w$ resolves two items $x,y \in V \cup E$ if $d(w,x) \neq d(w,y)$. Mixed resolving set $S \subseteq V$ is described as set such that for any pair of elements from $V \cup E$ there is some element in $S$ that resolves them. Following the earlier definitions, mixed metric basis and mixed metric dimension are defined as minimal mixed resolving set and cardinality of such minimal resolving set, respectively. Mixed metric dimension will be denoted as $\beta_M(G)$.

\begin{exm}Graph $G$ from Figure $\ref{smallgraph}$ is small graph with 5 vertices and 7 edges. The mixed metric dimension is equal 5, while metric dimension and edge metric dimension are equal 3 and 4, respectively. The mixed metric basis is $\{v_1,v_2,v_3,v_4,v_5\}$, i.e.  all vertices are elements of the basis and deleting any element from basis,  it will always exist two items with the same coordinates. All presented data is obtained by a total enumeration.
\end{exm}

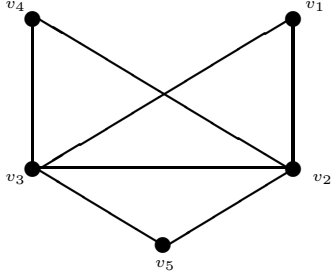
\begin{figure}[htbp]
\label{smallgraph}
\centering\setlength\unitlength{1mm}
\begin{picture}(50,44)
\thicklines
\tiny
\put(40.2,35.0){\line(0,-1){20.0}} 
\put(41.8,36.5){$v_1$}
\put(40.2,15.0){\line(-5,-3){17.3}} 

\put(42.8,13.5){$v_2$}
\put(22.9,5.0){\line(-5,3){17.3}} 
\put(21.9,2.0){$v_5$}
\put(5.6,15.0){\circle*{2}} 
\put(5.6,15.0){\line(0,1){20.0}} 
\put(2.0,13.5){$v_3$}
\put(40.0,35.2){\line(-5,-3){34.0}} 
\put(5.5,15.3){\line(2,0){34.0}} 
\put(2.0,36.5){$v_4$}

\put(40,14.9){\line(-5,3){34}} 
\put(40.2,35.0){\circle*{2}}

\put(40.2,15.0){\circle*{2}} 
\put(22.9,5.0){\circle*{2}}

\put(5.6,35.0){\circle*{2}}
\end{picture}
\caption{Small graph $G$ with 5 vertices}
\end{figure}

In the article \cite{yer17} where problem of mixed metric dimension was introduced, the authors presented some facts considering structure of mixed resolving sets. In cases were upper and lower bounds of mixed metric dimension could be easily obtained, the authors presented characterization of graphs whose mixed metric dimension reaches these bounds. Since lower and upper bounds are 2 and $n$, the authors have shown that for paths mixed metric dimension is 2, and mixed metric dimension is equal $n$ if and only if every vertex has a maximal neighbor. In order to better present mixed metric dimensions, the authors determined exact values for some classes of graphs, notably cycles, trees, complete bipartite graphs and also grid graphs. Moreover, few general lower/upper bounds are presented. Finally, the authors proved that computing mixed metric dimension is NP-hard in general case.

Since, the ILP formulation of the mixed metric problem is given in the paper \cite{yer17}, it naturally produces another lower bound, notably the LP relaxation of the given problem.

Some other classes of graphs also attracted attention of researchers, so among the literature it can be found Raza et al. (2019) which in \cite{raz19} gave the exact value of mixed metric dimension for three well-known classes of graphs: prism, antiprism and graph of convex polytope $R_n$. And similarly, Milivojevi\'c Danas in \cite{mil20}, provided the exact results for two other important well-known classes of graphs: flower snarks and wheels.

\subsection{Definitions and properties}

Let us denote $deg_v$ as degree of vertex $v \in V$ and
$\delta(G)$ as the minimum degree of vertices in $G$,
i.e. $\delta(G) = \min \{deg_v | v \in V\} $.

Following propositions, theorems and their corollaries will be used in the next section. We present them as they were stated in articles \cite{yer17,fil19}.

\begin{prp} \mbox{\rm(\cite{yer17})} \label{alldim} For any graph $G$ it holds
 \begin{equation}\label{Y01}\beta_{M}(G)\geq \max\{\beta(G),\beta_{E}(G)\}. \end{equation}
\end{prp}

For two vertices $u, v \in G$, we will say that they are \textit{false twins} if they have the same open neighborhoods, i.e., $N(u) = N(v)$. On the other hand,
the vertices $u, v$ are called \textit{true twins} if $N[u] = N[v]$. Also, vertex $v$ is called an \textit{extreme vertex} if $N(v)$ induces a complete graph.

\begin{pro}\mbox{\rm(\cite{yer17})} \label{edim4}
If $u, v$ are true twins in a graph $G$, then $u, v$ belong to every mixed metric generator for $G.$
\end{pro}
\begin{pro}\mbox{\rm(\cite{yer17})} \label{edim5}
If $u, v$ are false twins in a graph $G$ and $S$ is a mixed metric generator for $G$, then $\{u,v\}\cap S\neq \emptyset.$
\end{pro}
\begin{pro}\mbox{\rm(\cite{yer17})} \label{edim6}
If $u$ is  a simplicial (extreme) vertex in a graph $G$, then $u$ belongs to every mixed metric generator for $G$.
\end{pro}
\begin{cor}\mbox{\rm(\cite{yer17})} \label{edim7}
If $u$ is a vertex of degree 1 in a graph $G$, then $u$ belongs to every mixed metric generator for $G$.
\end{cor}

\begin{pro}\mbox{\rm(\cite{kel18})} \label{edim2} Let $G$ be a connected graph and let $\triangle(G)$ be the maximum degree of $G$. Then \begin{equation}\label{Y04}\beta_E(G)\geq \lceil\log_2\triangle (G)\rceil.\end{equation}
\end{pro}

\begin{thm}\label{lb2} \mbox{\rm(\cite{fil19})} \label{edim1} Let $G$ be a connected graph, then \begin{equation}\label{Y02}\beta_E(G) \geq 1 + \lceil \log_2 \delta(G) \rceil. \end{equation}\end{thm}

\section{New general lower bounds}

First, we will introduce a lower bound for mixed metric dimension of any connected graph $G$, which
slightly improves the lower bound for edge metric dimension from Theorem \ref{edim1} given in \cite{fil19}.

\begin{thm} \label{lb1} Let $G$ be a connected graph and let $x$ be an arbitrary vertex from mixed resolving set $S$ of graph $G$. Then,
$|S| \geq 1 + \lceil \log_2 (1+\deg_x) \rceil$.\end{thm}
\begin{proof}
Without loss of generality we can assume that $S = \{x, w_2, ..., w_p\}$.
Vector of metric coordinates of vertex $x$  with respect
to $S$ is $r(x,S) = (0,d_2,...,d_p)$, where $d_i = d(w_i,x)$, for all $i, 2 \le i \le p$.
Vertex $x$ is incident to $deg_x$ edges.
Name them as $e_1$, ..., $e_{deg_x}$.
For each position $i=2,...,p$ in the ordering of $S$ and each index $j=1,\ldots,deg_x$, edge $e_j$ is incident to vertex $x$, so by the definition
of $d(e_j,w_i)$ it directly implies $d(e_j,w_i) \in \{d_i-1, d_i\}$, i.e.
there can be only two different possible distances.
Therefore vertex $x$ and edges $e_1$, ..., $e_{deg_x}$ have
at most $2^{p-1}$ different mixed metric representations with respect to $S$,
implying $1+deg_x \leq 2^{p-1} \Rightarrow p \geq 1 + \log_2 (1+\deg_x)$.
Having in mind that $p=|S|$ is an integer, we have
$|S| \geq 1 + \lceil \log_2 (1+\deg_x) \rceil$.
\end{proof}

The following corollary describes vertices which can not be member of mixed resolving sets:

\begin{cor} \label{lb1rs} Let $G$ be a connected graph and let $v$ be an arbitrary vertex $v \in V$.
If $deg_v > 2^{\beta_M(G)-1} - 1$ then $v$ is not a member of any mixed resolving set $S$ of cardinality $\beta_{M}(G)$ of graph $G$.\end{cor}

Since $\delta(G)$ is the minimum degree of vertices in $G$ we have another corollary of Theorem \ref{lb1}:

\begin{cor} \label{lb1de} Let $G$ be a connected graph, then \begin{equation}\label{Nlb1}\beta_M(G) \geq 1 + \lceil \log_2 (\delta(G)+1) \rceil.
\end{equation}
\end{cor}

\begin{exm} This lower bound from Corollary $\ref{lb1rs}$, for the graph from the Figure  $\ref{smallgraph}$,  can be easily calculated. Since, the minimum degree of vertices $\delta(G)=2$ and by replacing in inequality ($\ref{Nlb1}$), the lower bound is  equal 3.
\end{exm}

For regular graphs, the following corollary holds:

\begin{cor}  Let $G$ be an $r$-regular graph, then \begin{equation}\beta_M(G) \geq 1 + \lceil \log_2 (r+1) \rceil.\end{equation}\end{cor}

We will now define the hitting set $H$. For a given set $U$ and a collection $\mathcal{T}$ of subsets $S_1,...,S_m$ of $U$ such that their union is equal to $U$, the hitting set $H$ is a set which has a nonempty intersection with each set from this collection, i.e. $(\forall i\in \{1,...,m\})\, H\cap S_i\neq \emptyset$. Finding a hitting set of minimum cardinality is called the hitting set problem. This problem is equivalent to the set covering problem, because for the given set $U$ and collection of subsets $S_1,...,S_m$, the set covering problem is finding  subcollection $\mathcal{C} \subseteq \mathcal{T}$ of minimum cardinality  whose union is $U$.

For an arbitrary edge $uv$ we define sets $V_{<,uv}=\{w \in V| d(u,w)<d(v,w)\}$ and  $V_{>,uv}=\{w \in V| d(u,w)>d(v,w)\}$. The relationship between these sets and mixed resolving sets are given in the following lemma.

\begin{lem} \label{caes}
Let $G$ be a connected graph, $e\in E$ an arbitrary edge and $S$ a mixed resolving set, then
 \begin{itemize}
  \item [a)] $V_{>,e}\cap S\neq 0;$
   \item [b)] $V_{<,e}\cap S\neq 0$.

  \end{itemize}
\end{lem}

\begin{proof}

a) Suppose the opposite, i.e. $(\exists uv\in E)$ so that $V_{>,uv}\cap S=\emptyset$, which means that for each vertex $w$ holds $ d(u,w)\leq d(v,w)$. According to equality (\ref{dis}), it is easy to see that mixed metric coordinate of vertex $u$ is same as the mixed metric coordinate of edge $uv$, i.e. $r(u,S)=r(uv,S)$, which means that $S$ is not a mixed resolving set which is contradiction to the starting assumption.

b) The proof of this part of lemma is analogous to the proof of part a).
\end{proof}

For each edge $e\in E$, we assign sets $V_{<,e}$ and $V_{>,e}$ and it is easy to see that there are $2m$ sets. Our idea is to find a minimal hitting set $H^{*}$ for the family of sets $\{V_{>,e},V_{<,e}|e\in E\}$. The cardinality of minimum hitting set of this family of sets will be denoted as $MHSP(\{V_{<,e},V_{>,e}|e\in E\})$.

In the following theorem the relationship between mixed metric dimension and the minimum hitting set problem will be established.

\begin{thm} \label{lb1reg}For any connected graph $G$, it holds \begin{equation}\label{Nlb2}\beta_M (G)\geq MHSP(\{V_{<,e},V_{>,e}|e\in E\}). \end{equation}

\end{thm}

\begin{proof}

Let $\beta_M(G)$ be the mixed metric dimension of an arbitrary graph $G$.  Then, there is a mixed resolving set $S$ so that $|S|=\beta_M (G).$ From Lemma $\ref{caes}$, for arbitrary edge $e$, it follows that

\begin{equation} \label{presek}
V_{<,e}\cap S\neq \emptyset \wedge V_{>,e}\cap S\neq \emptyset.
\end{equation}
This means that there is at least one element from $S$ in each of these sets for every edge $e$.

Let us now consider minimal hitting set problem over family of these sets $V_{<,e}$ and $V_{>,e}$ where $e \in E$. Since mixed resolving set $S$ satisfies (\ref{presek}), $S$ is a hitting set for a family of sets  $\{e \in E | V_{>,e},V_{<,e}\}$. The cardinality of each hitting set is greater or equal to the cardinality of minimal hitting set, so it can be concluded that inequality ($\ref{Nlb2}$) holds.
\end{proof}

\begin{exm}The lower bound from Theorem $\ref{lb1reg}$, for the graph from the Figure  $\ref{smallgraph}$, is obtained using total enumeration for hitting set problem and it is equal 5.\end{exm}

Let us present another lower bound for mixed metric dimension based on the diameter of a graph.

\begin{thm}\label{lb2reg}  Let $G=(V,E)$ be a connected graph with mixed metric dimension $\beta_M (G)$  and let $D$ be the diameter of graph $G$, then
 \begin{equation}\label{Y03}|V|+|E|\leq D^{\beta_M (G)}+\beta_M (G)(\bigtriangleup (G)+1). \end{equation}
\end{thm}
\begin{proof}
We will consider all possible representations of metric coordinates for  all items of the graph $G$. Since the diameter of graph is $D$, then it is easy to see that each item of graph can have integer coordinates between 0 and $D$. The set of all items can be separated into two disjunctive classes:
\begin{itemize}
  \item [$I)$] items whose metric coordinates do not have 0;
  \item [$II)$] items whose metric coordinates have 0.
\end{itemize}

 Each item from $I)$ class, which does not have a coordinate equal 0, must have a unique coordinates from one of $D^{\beta_M (G)}$ possibilities. For items from $II)$, i.e. with one coordinate equal to zero, it is easy to see that it will be the vertex which is an element of the basis,  or edge containing that vertex. Hence, for each element of the basis, there are at the most $\triangle (G)+1$ possibilities, i.e. it must have a unique coordinate from one of $\beta_M (G)(\triangle (G)+1)$ possibilities. Therefore, from the previous, it is easy to conclude that inequality $(\ref{Y03})$ follows, thus completing the proof of theorem.

\end{proof}

\begin{exm}The lower bound from Theorem $\ref{lb2reg}$, for the graph from the Figure  $\ref{smallgraph}$, is obtained calculating
inequalities ($\ref{Y03}$) and it is equal 2.
\end{exm}

\section{Exact results on torus graph}

In this section we will use previously introduced general lower bounds to obtain the exact values of mixed metric dimension of torus graph.

\begin{thm} For $m,n \geq 3$ it holds $\beta_M(T_{m,n}) = 4$.\end{thm}
\begin{proof}  \textbf{\underline{Step 1}:}  {\em Upper bound is 4}.  \\

There are four cases:\\
\textbf{Case 1.} $m=2k+1, n=2l+1$\\
Let $S = \{(0,0), (0,l), (1,l+1), (k+1,l+1)\}$. Let us prove that $S$ is mixed metric resolving set. The representation of coordinates of each vertex and each edge,with respect to $S$, is shown in the Table \ref{vtor2} and Table \ref{etor2}.

\begin{table}
\tiny
\begin{center}
\caption{  Metric coordinates of vertices of $T_{2k+1,2l+1}$}
\label{vtor2}
\begin{tabular}{|c|c|c|}
 \hline
  vetex & cond. & $r(v,S)$\\
 \hline
$(0,0)$ &   & ($0, l, l+1, l+k$) \\
$(i,0)$ &  $1\leq i \leq k$ & ($i,i+l,l+i-1,l+k-i+1$) \\
$(0,j)$ &  $1 \leq j \leq l$ & ($j,l-j, l-j+2, l-j+k+1$) \\
$(i,j)$ & $1\leq i \leq k-1$  &($j+i,i+l-j,l-j+i,l-j+k-i+2$)\\
&$1\leq j \leq l$  &  \\
$(0,j)$  &  $l+1 \leq j \leq n-1$ & ($n-j,j-l,j-l,j-l+k-1$) \\

$(i,j)$ & $l+1\leq j \leq n-1$  &($n-j+i,j-l+i,j-l+i-1,k-i+j-l$)\\
&$1\leq i \leq k-1$  &  \\

$(i,0)$ &  $k+2 \leq i \leq m-1$ & ($m-i, m-i+l, m-i+l+1, l-k+i-1$) \\

$(i,j)$ & $k+2\leq i \leq m-1$  &($m-i+j,m-i+l-j,m-i+l-j+2,i-k+l-j$)\\
&$1\leq j \leq l$  &  \\

$(i,j)$ & $k+2\leq i \leq m-1$  &($m+n-i-j,j-l+m-i,,$\\
&$l+1\leq j \leq n-1$  & $j-l-i+m,j-l+i-k-2)$ \\

$(k+1,0)$ &  & ($m-i, m-i+l, k+l,l$) \\
$(k,j)$ &  $1 \leq j \leq l$ & ($k+j,k+l-j,k+l-j,l-j+2$) \\
$(k+1,j)$ &  $1 \leq j \leq l$ & ($k+j, k+l-j, k+l-j+1, l-j+1$) \\
$(k,j)$ &  $l+1 \leq j \leq n-1$ & ($k+n-j, j-l+k, k+j-l-2, j-l$) \\
$(k+1,j)$ &  $l+1 \leq j \leq n-1$ & ($k+n-j,k+j-l,k+j-l-1,j-l-1$) \\
	  \hline
\end{tabular}
\end{center}
\end{table}

\begin{table}
\tiny
\begin{center}
\caption{  Metric coordinates of edges of $T_{2k+1,2l+1}$}
\label{etor2}
\begin{tabular}{|c|c|c|}

 \hline
  edge & cond. & $r(e,S)$\\
 \hline

 $(0,0)(1,0)$ &  & ($0,l,l,k+l$) \\
 $(0,0)(0,n-1)$ &  & ($0,l,l+1,k+l-1$) \\
 $(0,0)(m-1,0)$ &  & ($0,l,l+1,l+k-1$) \\
$(0,j)(0,j+1)$ & $0\leq j \leq l-1$   & ($j, l-j-1, l-j+1, k+l-j$) \\
$(i,0)(i+1,0)$ &  $1\leq i \leq k-1$ & ($i,l+i,l+i-1,k-i+l$) \\
$(i,j)(i,j+1)$ &  $1\leq i \leq k$ & ($i+j,l-j-1+i,l+i-j-1,l-j+k-i+1$) \\
&  $1\leq j \leq l-1$ &  \\
$(i,l)(i,l+1)$ &  $1\leq i \leq k$ & ($l+i,i,i-1,k-i+1$) \\
$(i,0)(i,1)$ &  $1\leq i \leq k$ & ($i,l+i-1,l+i-1,l+k-i+1$) \\
$(0,j)(1,j)$ &  $1\leq j \leq l$ & ($j,l-j,l-j+1,k+l-j+1$) \\
$(0,j)(0,j+1)$ &  $l+1\leq j \leq n-2$ & ($n-j-1,j-l,j-l,k+j-l-1$) \\
$(i,j)(i,j+1)$ &  $1\leq i \leq k$ & ($n-j+i-1,j-l+i,$ \\
&  $l+1\leq j \leq n-2$ & $j-l+i-2,k+j-l-i$) \\
$(i,j)(i+1,j)$ &  $1\leq i \leq k-1$ & ($i+j,l-j+i,l-j+i,l-j+k-i+1$) \\

&  $1\leq j \leq l$ &  \\
$(i,0)(i,n-1)$ &  $1\leq i \leq k$ & ($i,l+i,l+i-1,l+k-i$) \\
$(i,j)(i+1,j)$ &  $1\leq i \leq k-1$ & ($n-j+i,j-l+i,$ \\
&  $l+1\leq j \leq n-1$ &  $j-l+i-2,k-i+j-l-1$)\\
$(0,j)(1,j)$ &  $l+1\leq j \leq n-1$ & ($n-j,j-l,j-l-1,k+j-l-1$) \\
$(i,j)(i+1,j)$ &  $k+1\leq i \leq m-2$ & ($j+m-i-1,l-j+m-i-1,$ \\
&  $1\leq j \leq l$ & $m-i+l-j+1,l+1-j+i-k-1$)\\
$(i,j)(i,j+1)$ &  $k+2\leq i \leq m-1$ & ($m-i+j,l-j-1+m-i,$ \\
&  $1\leq j \leq l-1$ &$l-j+m-i,l-j+i-k-1$) \\
$(i,0)(i+1,0)$ &  $k+1\leq i \leq m-2$ & ($m-i-1,m-i+l-1,m-i+l,i-k-1+l$) \\
$(i,0)(i,1)$ &  $k+2\leq i \leq m-1$ & ($m-i,m-i+l-1,m-i+l+1,l+i-j-k-1$) \\
$(k,j)(k+1,j)$ &  $1\leq j \leq l$ & ($k+j,k+l-j,k+l-j,l-j+1$)\\
$(0,j)(m-1,j)$ &  $1\leq j \leq l$ & ($j,l-j,l-j+2,k+l-j$) \\
$(k+1,j)(k+1,j+1)$ &  $1\leq j \leq l-1$ & ($k+j,k+l-j-1,k+l-j,l-j$)\\
$(i,j)(i+1,j)$ &  $k+1\leq i \leq m-2$ & ($n-j+m-i-1,m-i+j-l-1,$\\
&  $l+1\leq j \leq n-1$ & $m-i+j-l-1,i-k+j-l-2)$ \\
$(i,j)(i,j+1)$ &  $k+2\leq i \leq m-1$ & ($n-j-1+m-i,j-l+m-i,$\\
&  $l+1\leq j \leq n-2$ & $j-l+m-i,j-l+i-k-2)$ \\
$(i,l)(i,l+1)$ &  $k+1\leq i \leq m-1$ & ($m-i+l,m-i,m-i+1,i-k-1)$\\
$(k,j)(k+1,j)$ &  $l+1\leq j \leq n-1$ & ($n-j+k,j-l+k,k+j-l-2,j-l-1)$\\
$(i,0)(i,n-1)$ &  $k+2\leq i \leq m-1$ & ($m-i,m-i+l,m-i+l,l+i-k-2)$\\
$(0,j)(m-1,j)$ &  $l+1\leq j \leq n-1$ & ($n-j,j-l,j-l,k+j-l-2)$\\
$(0,l)(0,l+1)$ &  & ($l,0,1,k)$\\
$(0,l)(m-1,l)$ &  & ($l,0,2,k)$\\
$(k,0)(k+1,0)$ &  & ($k,k+l,k+l-1,l)$\\
$(k+1,0)(k+1,1)$ &  & ($k,k+l-1,k+l,l)$\\	
$(k+1,0)(k+1,n-1)$ &  & ($k,k+l,k+l-1,l-1)$\\
$(k+1,j)(k+1,j+1)$ &  $l+1\leq j \leq n-2$ & ($n-j+k-1,j-l+k,j-l+k-1,j-l-1)$\\
$(k+1,l)(k+1,l+1)$ &  & ($k+l,k,k,0)$\\

\hline
\end{tabular}
\end{center}
\end{table}

\newpage

Since metric coordinates of all items are mutually different, $S$ is a mixed resolving set. Therefore, $\beta_{M}(T_{2k+1,2l+1})\leq 4.$

\textbf{Case 2.} $m=2k+1, n=2l$\\
Let $S = \{(0,0), (0,l), (1,0), (k+1,1)\}$. Let us prove that $S$ is mixed metric resolving set. The representation of coordinates of each vertex and each edge, with respect to $S$, is shown in the Table \ref{vtor7} and Table \ref{vtor6}.
\begin{table}
\tiny
\begin{center}
\caption{  Metric coordinates of vertices of $T_{2k+1,2l}$}
\label{vtor7}
\begin{tabular}{|c|c|c|}
 \hline
  vetex & cond. & $r(v,S)$\\
 \hline
$(0,0)$ &   & ($0, l, 1, k+1$) \\
$(i,0)$ &  $1\leq i \leq k$ & ($i,i+l,i-1,k-i+2$) \\
$(0,j)$ &  $1 \leq j \leq l$ & ($j,l-j, j+1, k+j-1$) \\
$(i,j)$ & $1\leq i \leq k$  &($j+i,l-j+i,j+i-1,j+k-i$)\\
&$1\leq j \leq l$  &  \\
$(0,j)$  &  $l+1 \leq j \leq n-1$ & ($n-j,j-l,n-j+1,n-j+k+1$) \\

$(i,j)$ & $l+1\leq j \leq n-1$  &($n-j+i,j-l+i,n-j+i-1,n-j+k-i+2$)\\
&$1\leq i \leq k$  &  \\

$(i,0)$ &  $k+2 \leq i \leq m-1$ & ($m-i, m-i+l, m-i+1, i-k$) \\

$(k+1,0)$ &   & ($k, k+l, k, 1$) \\

$(i,j)$ & $k+2\leq i \leq m-1$  &($m-i+j,m-i-j+l,m-i+j+1,i-k+j-2$)\\
&$1\leq j \leq l$  &  \\
$(k+1,j)$ & $1\leq j \leq l$   & ($m-i+j, m-i+l-j,k+j ,j-1$) \\

$(i,j)$ & $k+2\leq i \leq m-1$  &($m-i+n-j,m-i+j-l,$\\
&$l+1\leq j \leq n-1$  & $m-i+1+n-j,n+i-k-j)$ \\
$(k+1,j)$ & $l+1\leq j \leq n-1$   & ($n-j+k,k+j-l,n-j+k ,n-j+1$) \\

	  \hline
\end{tabular}
\end{center}
\end{table}

\begin{table}
\tiny
\begin{center}
\caption{  Metric coordinates of edges of $T_{2k+1,2l}$}
\label{vtor6}
\begin{tabular}{|c|c|c|}

 \hline
  edge & cond. & $r(e,S)$\\
 \hline
 $(0,0)(0,1)$ &  & ($0,l-1,1,k$) \\
 $(0,0)(1,0)$ &  & ($0,l,0,k+1$) \\
 $(0,0)(0,n-1)$ &  & ($0,l-1,1,k+1$) \\
 $(0,0)(m-1,0)$ &  & ($0,l,1,k$) \\
$(0,j)(0,j+1)$ & $1\leq j \leq l-1$   & ($j, l-j-1, j+1, k+j-1$) \\
$(i,0)(i+1,0)$ &  $1\leq i \leq k$ & ($i,i+l,i-1,k-i+1$) \\
$(i,j)(i,j+1)$ &  $1\leq i \leq k$ & ($i+j,i+l-j-1,j+i-1,k+j-i$) \\
&  $1\leq j \leq l-1$ &  \\
$(i,l)(i,l+1)$ &  $1\leq i \leq k$ & ($l+i-1,i,l+i-2,k-i+l$) \\
$(i,0)(i,1)$ &  $1\leq i \leq k$ & ($i,l-1+i,i-1,k-i+1$) \\
$(0,j)(1,j)$ &  $1\leq j \leq l$ & ($j,l-j,j,k+j-1$) \\
$(0,j)(1,j)$ &  $1\leq j \leq l$ & ($j,l-j,j,k+j-1$) \\
$(i,j)(i,j+1)$ &  $1\leq i \leq k$ & ($n-j+i-1,n-j+i,$ \\
&  $l+1\leq j \leq n-2$ & $j-l+i-1,n-j+k-i-1$) \\
$(0,j)(0,j+1)$ &  $l+1\leq j \leq n-2$ & ($n-j-1,j-l,n-j,k+n-j$) \\
$(i,j)(i,j+1)$ &  $1\leq i \leq k$ & ($n-j-1+i,j-l+i,n-j-2+i,k+n-j-i+1$) \\
&  $l+1\leq j \leq n-2$ &  \\
$(i,j)(i+1,j)$ &  $1\leq i \leq k$ & ($i+j,i+l-j,j+i-1,k-i+j-1$) \\
&  $1\leq j \leq l$ &  \\
$(i,0)(i,n-1)$ &  $1\leq i \leq k$ & ($i,i+l-1,i-1,k-i+2$) \\
$(i,j)(i+1,j)$ &  $1\leq i \leq k-1$ & ($n-j+i,j-l+i,$ \\
&  $l+1\leq j \leq n-1$ &  $n-j+i-1,n+k-i-j+1$)\\
$(0,j)(1,j)$ &  $l+1\leq j \leq n-1$ & ($n-j,j-l,n-j,n-j+k+1$) \\
$(i,j)(i+1,j)$ &  $k+1\leq i \leq m-2$ & ($m-i+j-1,m-i+l-j-1,$ \\
&  $1\leq j \leq l$ & $m-i+j,i-k+j-2$)\\
$(i,j)(i,j+1)$ &  $k+2\leq i \leq m-1$ & ($m-i+j,m-i+l-j-1,$ \\
&  $1\leq j \leq l-1$ &$m-i+1+j,i-k+j-2$) \\

$(k+1,j)(k+1,j+1)$ &  $1\leq j \leq l-1$ & ($m-i+j,m-i+l-j-1,k+j,i-k+j-2$) \\
$(i,0)(i+1,0)$ &  $k+1\leq i \leq m-2$ & ($m-i-1,m-i+l-1,m-i,i-k$) \\
$(i,0)(i,1)$ &  $k+2\leq i \leq m-1$ & ($m-i,m-i-1+l,m-i+1,i-k-1$) \\
 $(k+1,0)(k+1,1)$ &  & ($k,k+l-1,k,0$) \\
$(0,j)(m-1,j)$ &  $1\leq j \leq l$ & ($j,l-j,j+1,k+j-2$) \\
$(i,j)(i+1,j)$ &  $k+1\leq i \leq m-2$ & ($n-j+m-i-1,j-l+m-i-1,$\\
&  $l+1\leq j \leq n-1$ & $m+n-j-i,n-j+i-k)$ \\
$(i,j)(i,j+1)$ &  $k+2\leq i \leq m-1$ & ($n-j+m-i-1,m+j-i-l,$\\
&  $l+1\leq j \leq n-2$ & $n-j+m-i,n-j+i-k-1)$ \\
$(k+1,j)(k+1,j+1)$ &  $l+1\leq j \leq n-2$ & ($n-j+m-i-1,j-l+m-i,$\\
&   & $k+n-j-1,n-j+i-k-1$) \\
$(i,l)(i,l+1)$ &  $k+2\leq i \leq m-1$ & ($l+m-i-1,m-i,l+m-i,l+i-k-2$\\
$(k+1,l)(k+1,l+1)$ &   & ($l+k-1,k,l+k-1,l-1)$\\
$(i,0)(i,n-1)$ &  $k+2\leq i \leq m-1$ & ($m-i,m-i+l,m-i+1,i-k)$\\
$(k+1,0)(k+1,n-1)$ &   & ($k,l+k-1,k,1)$\\

$(0,j)(m-1,j)$ &  $l+1\leq j \leq n-1$ & ($n-j,j-l,n-j+1,k+n-j)$\\

$(0,l)(0,l+1)$ &  & ($l-1,0,l,k+l-1)$\\
$(k+1,0)(k+2,0)$ &  & ($k,k+l-1,k,0)$\\

	  \hline
\end{tabular}
\end{center}
\end{table}

Since metric coordinates of all items   are mutually different, so $S$ is a mixed resolving set. Therefore, $\beta_{M}(T_{2k+1,2l})\leq 4.$

\newpage

\textbf{Case 3.} $m=2k, n=2l+1$\\
Let $S =\{(0,0),(k,0),(0,1),(1,l+1)\}$. Since $C_m \Box C_n$ is the same as $C_n \Box C_m$, the proof of this case is similar to the proof of Case 2.\\

\textbf{Case 4.} $m=2k, n=2l$\\
Let $S = \{(0,0), (0,1), (1,l), (k,0)\}$. Let us prove that $S$ is mixed metric resolving set. The representation of coordinates of each vertex and each edge, with respect to $S$, is shown in the Table \ref{etor} and Table \ref{vtor}.

\begin{table}
\tiny
\begin{center}
\caption{  Metric coordinates of vertices of $T_{2k,2l}$}
\label{etor}
\begin{tabular}{|c|c|c|}
 \hline
  vetex & cond. & $r(v,S)$\\
 \hline
$(0,0)$ &   & ($0, 1, l+1, k$) \\
$(i,0)$ &  $1\leq i \leq k$ & ($i,i+1,l+i-1,k-i$) \\
$(0,j)$ &  $1 \leq j \leq l$ & ($j,j-1, l-j+1, k+j$) \\
$(i,j)$ & $1\leq i \leq k$  &($j+i,j-1+i,l-j+i-1,j+k-i$)\\
&$1\leq j \leq l$  &  \\
$(0,j)$  &  $l+1 \leq j \leq n-1$ & ($n-j,n-j+1,j-l+1,n-j+k$) \\

$(i,j)$ & $l+1\leq j \leq n-1$  &($n-j+i,n-j+i+1,j-l+i-1,n-j+k-i$)\\
&$1\leq i \leq k$  &  \\

$(i,0)$ &  $k+1 \leq i \leq m-1$ & ($m-i, m-i+1, m-i+l+1, i-k$) \\

$(i,j)$ & $k+1\leq i \leq m-1$  &($m-i+j,m-i+j-1,m-i+l-j+1,i-k+j$)\\
&$1\leq j \leq l$  &  \\

$(i,j)$ & $k+1\leq i \leq m-1$  &($m+n-i-j,m+n-i-j+1,$\\
&$l+1\leq j \leq n-1$  & $m+j-l-i+1,n+i-k-j)$ \\

	  \hline
\end{tabular}
\end{center}
\end{table}

\begin{table}
\tiny
\begin{center}
\caption{  Metric coordinates of edges of $T_{2k,2l}$}
\label{vtor}
\begin{tabular}{|c|c|c|}

 \hline
  edge & cond. & $r(e,S)$\\
 \hline
 $(0,0)(0,1)$ &  & ($0,0,l,k$) \\
 $(0,0)(1,0)$ &  & ($0,1,l,k-1$) \\
 $(0,0)(0,n-1)$ &  & ($0,1,l,k$) \\
 $(0,0)(m-1,0)$ &  & ($0,1,l+1,k-1$) \\
$(0,j)(0,j+1)$ & $1\leq j \leq l-1$   & ($j, j-1, l-j, k+j$) \\
$(i,0)(i+1,0)$ &  $1\leq i \leq k-1$ & ($i,i+1,l+i-1,k-i-1$) \\
$(i,j)(i,j+1)$ &  $1\leq i \leq k$ & ($i+j,i+j-1,l-j+i-2,k+j-i$) \\
&  $1\leq j \leq l-1$ &  \\
$(i,0)(i,1)$ &  $1\leq i \leq k$ & ($i,i,l+i-2,k-i$) \\
$(0,j)(1,j)$ &  $1\leq j \leq l$ & ($j,j-1,l-j,k+j-1$) \\
$(0,j)(0,j+1)$ &  $l+1\leq j \leq n-2$ & ($n-j-1,n-j,j-l+1,n-j-1+k$) \\

$(i,j)(i,j+1)$ &  $1\leq i \leq k$ & ($n-j+i-1,n-j+i,$ \\
&  $l+1\leq j \leq n-2$ & $j-l+i-1,n-j+k-i-1$) \\

$(i,j)(i+1,j)$ &  $1\leq i \leq k-1$ & ($i+j,i+j-1,l-j+i-1,j+k-i-1$) \\
&  $1\leq j \leq l$ &  \\

$(i,l)(i,l+1)$ &  $1\leq i \leq k$ & ($l+i-1,l+i-1,i-1,l+k-i-1$) \\

$(i,0)(i,n-1)$ &  $1\leq i \leq k$ & ($i,i+1,l+i-2,k-i$) \\
$(i,j)(i+1,j)$ &  $1\leq i \leq k-1$ & ($n-j+i,n-j+i+1,$ \\
&  $l+1\leq j \leq n-1$ &  $j-l+i-1,n-l+k-i-1$)\\
$(0,j)(1,j)$ &  $l+1\leq j \leq n-1$ & ($n-j,n-j+1,j-l,n-j+k-1$) \\
$(i,j)(i+1,j)$ &  $k+1\leq i \leq m-2$ & ($m-i+j-1,m-i+j-2,$ \\
&  $1\leq j \leq l$ & $m-i+l-j,i-k+j$)\\
$(i,j)(i,j+1)$ &  $k+1\leq i \leq m-1$ & ($m-i+j,m-i+j-1,$ \\
&  $1\leq j \leq l-1$ &$,m-i+l-j,j+i-k$) \\
$(i,0)(i+1,0)$ &  $k+1\leq i \leq m-2$ & ($m-i-1,m-i,m-i+l,i-k$) \\

$(i,0)(i,1)$ &  $k+1\leq i \leq m-1$ & ($m-i,m-i,m-i+l,i-k$) \\
$(k,j)(k+1,j)$ &  $1\leq j \leq l$ & ($m-k+j-1,m-k+j-2,$\\
&   & $,m-k+l-j-1,j$)  \\
$(0,j)(m-1,j)$ &  $1\leq j \leq l$ & ($j,j-1,l-j+1,m-k+j-1$) \\
$(k+1,j)(k+1,j+1)$ &  $1\leq j \leq l-1$ & ($m-k+j-1,m-k+j-2,$\\
&   & $,m-k+l-j-1,j+1$) \\
$(i,j)(i+1,j)$ &  $k+1\leq i \leq m-2$ & ($n-j+m-i-1,n+m-j-i,$\\
&  $l+1\leq j \leq n-1$ & $j-l+m-i,n-j+i-k)$ \\
$(i,j)(i,j+1)$ &  $k+1\leq i \leq m-1$ & ($n-j+m-i-1,n+m-j-i,$\\
&  $l+1\leq j \leq n-2$ & $j-l+m-i+1,n-j+i-k-1)$ \\
$(i,l)(i,l+1)$ &  $k+1\leq i \leq m-1$ & ($l+m-i-1,m+l-i-1,$\\
&   & $,m-i+1,l+i-k-1)$  \\
$(k,j)(k+1,j)$ &  $l+1\leq j \leq n-1$ & ($n-j+m-i-1,m+n-j-i+1,$\\
&   & $,j-l+k-1,n-j)$  \\
$(i,0)(i,n-1)$ &  $k+1\leq i \leq m-1$ & ($m-i,m-i+1,m-i+l,i-k)$\\
$(0,j)(m-1,j)$ &  $l+1\leq j \leq n-1$ & ($n-j,n-j+1,j-l+1,n-j+k-1)$\\
$(0,l)(0,l+1)$ &  & ($l-1,l-1,1,l-1+k)$\\
$(0,l)(m-1,l)$ &  & ($l,l-1,1,k+l-1)$\\
$(k,0)(k+1,0)$ &  & ($k-1,k,m-k-1+l,0$\\

	  \hline
\end{tabular}
\end{center}
\end{table}
Since metric coordinates of all items  are mutually different, $S$ is a mixed resolving set. Therefore, $\beta_{M}(T_{2k,2l})\leq 4.$\\

\textbf{\underline{Step 2}:}  {\em Lower bound is 4}.  \\
Torus graph is $4$-regular graph, so by Corollary \ref{lb1reg} follows
$\beta_M(T_{m,n})  \geq 1 + \lceil log_2 (r+1) \rceil = 1 + \lceil log_2 5 \rceil = 4$.

Therefore, from the previous two steps, it follows that $\beta_M(T_{m,n})=4.$
\end{proof}

\section{Direct comparison of lower bounds}

In this section, it will be given direct comparison between lower bounds known in the literature (\cite{yer17},\cite{fil19}) with the new lower bounds obtained in this paper.

First, comparison will be performed on all connected graphs with 5 vertices. There are 21 such graphs and their graphic representations could be find at \url{https://mathworld.wolfram.com/ConnectedGraph.html}. The results in the Table $\ref{jedge2}$ are given in the same order as graphic representation on link and in this table are shown the comparisons of the various lower bounds for these graphs. In Table $\ref{jedge2}$, $|E|$ is  number of edges, $\beta(G)$ and $\beta_E(G)$ are metric dimension and edge metric dimension, respectively. In the following columns, L1 and L2
are the notation for lower bounds from Proposition 4 and Theorem 1. Each
of Proposition 1, Proposition 2, Proposition 3 and Corollary 1 determines
one lower bound. For the purpose of transparency of the Table 7, we have
decided to give a unified lower bound that encompasses all of them. It will be
denoted as L3. This lower bound cannot be obtained generally, while for each
specific graph, all three lower bounds from propositions and Corollary 1 can be
calculated separately and unified together. Lower bound L4 is a LP relaxation
of the mixed metric dimension problem. In the following columns new lower
bounds N1, N2 and N3 from Corollary 3, Theorem 3 and Theorem 4 are
given respectively.

It should be noted that a total enumeration is able to quickly compute metric dimension, edge metric dimension and mixed metric dimension
for graphs up to 36 vertices, so it is used to obtain data for $\beta(G)$, $\beta_E(G)$ and $\beta_M(G)$
in Table $\ref{jedge2}$ and Table $\ref{jedge}$. Data of column L4 in these tables, that represents a LP relaxation
of the mixed metric dimension problem, can be quickly obtained by any linear programming software:
CPLEX, Gurobi, GLPK, LP\_solve, etc. Data of column N2 is also computed by a total enumeration.

\begin{table}
\small
\caption{Direct comparison of lower bounds for connected graphs with 5 vertices }
\label{jedge2}
\begin{tabular}{|l|l|l|l|l|l|l|l|l|l|l|l|l|}
\hline

&\multicolumn{3}{|c|}{}&\multicolumn{4}{|c|}{LB from lit.}&\multicolumn{3}{|c|}{ New LB }&\\
 \hline
  Num& $|E|$ & $\beta(G)$ & $\beta_E(G) $& L1 &L2 & L3&L4&N1&N2&N3& $\beta_{M}(G)$ \\
\hline
\hline
1. &  4&3&3& 2 & 1 &$\underline{\textbf{4}}$  & $\underline{\textbf{4}}$ & 2 & $\underline{\textbf{4}}$  & 2 & 4\\
 \hline
 2.&  4&2&2 &2 &1 &$\underline{\textbf{3}}$  &$\underline{\textbf{3}}$  & 2&$\underline{\textbf{3}}$  &2&3 \\
 \hline
3.&  5 & 2 & 3  & 2 & 1 & $\underline{\textbf{4}}$  & $\underline{\textbf{4}}$ & 2 &$\underline{\textbf{4}}$ &2&  4\\
 \hline
  4. &  5 & 2 & 2 & 2 & 1 & $\underline{\textbf{3}}$  &$\underline{\textbf{3}}$ & 2 & $\underline{\textbf{3}}$  & 2 &3  \\
 \hline
 5.&5&2& 2&2 & 1 & 2& $\underline{\textbf{3}}$ & 2 &2& 2&3\\
 \hline
6. &6& 2& 3& 2&1& 3& $\underline{\textbf{4}}$ &  2&$\underline{\textbf{4}}$ & 2&4\\
 \hline
 7. &6&3&3 & 2& 2&2 &3 & 3 &2&2 &4 \\
 \hline
 8.  &7&3&4 & 2& 2&$\underline{\textbf{5}}$  & $\underline{\textbf{5}}$ &3  &$\underline{\textbf{5}}$ & 2&5\\
 \hline
9.  &  4 &1&1 & 1&1 & $\underline{\textbf{2}}$ &$\underline{\textbf{2}}$  & $\underline{\textbf{2}}$  &$\underline{\textbf{2}}$ &$\underline{\textbf{2}}$  &2\\
 \hline
 10.&  5&2&2 & 2 & 1 & $\underline{\textbf{3}}$ &$\underline{\textbf{3}}$ & 2 &$\underline{\textbf{3}}$ &2& 3 \\
 \hline
11.  &6&2&3 &2 &2 & $\underline{\textbf{4}}$ &$\underline{\textbf{4}}$  & 3 &$\underline{\textbf{4}}$ &2 &4\\
 \hline
  12.   &6&2&3 &2 &1 & $\underline{\textbf{4}}$ &$\underline{\textbf{4}}$ &  2&$\underline{\textbf{4}}$ &2 &4\\
 \hline
13. &7&3&3 & 2&1 &$\underline{\textbf{4}}$  &$\underline{\textbf{4}}$  & 2 &$\underline{\textbf{4}}$ &2&4 \\
 \hline
14. &5&2&2 & 1 & 2& 0& $\underline{\textbf{3}}$ &$\underline{\textbf{3}}$   &$\underline{\textbf{3}}$ & 2&3\\
 \hline
15. &6&2&2 &2 & 2& 1& $\underline{\textbf{3}}$ & $\underline{\textbf{3}}$  &$\underline{\textbf{3}}$ & 2&3\\
 \hline
16.&7&2&3 & 2& 2& 2& $\underline{\textbf{4}}$ & 3 &$\underline{\textbf{4}}$ &2 &4\\
 \hline

  17.&8&3& 4&2 &2 & $\underline{\textbf{5}}$ & $\underline{\textbf{5}}$ &3  &$\underline{\textbf{5}}$ &2&5 \\
 \hline
18.&7&2&3 & 2&2 &3 &$\underline{\textbf{4}}$  &3  &3&2&4 \\
 \hline
19. &8&2&$\underline{\textbf{4}}$  & 2&3 & 2&$\underline{\textbf{4}}$ & 3 &$\underline{\textbf{4}}$ &2 &4\\
 \hline
  20. &9&3& 4& 2&3 &$\underline{\textbf{5}}$  &$\underline{\textbf{5}}$  & 3 &$\underline{\textbf{5}}$ & 2&5\\
 \hline
  21.  &10&4& 4& 2& 3&$\underline{\textbf{5}}$  &$\underline{\textbf{5}}$  & 4& $\underline{\textbf{5}}$ &3& 5\\
 \hline
\end{tabular}
\end{table}

As it can be seen from Table \ref{jedge2} new lower bounds obtain better results than bounds L1, L2, L3 and L4 from the literature. In three of 21 cases mixed metric dimension was not reached by new lower bounds.

These results are not quite representative since graphs in question are with small number of vertices ($|V| = 5$). In order to improve comparison we conducted it on some well-known graphs.

Additional comparison will be conducted on 12 well-known graphs. In the Table $\ref{jedge}$ are shown graph characteristics for each graph, while the comparisons of the various lower bounds are shown in the Table  $\ref{jedge1}$. Columns in Table $\ref{jedge1}$, nominated as L1, L2, L3, L4, N1, N2 and N3, have the same meaning as in Table \ref{jedge2}. From the Table \ref{jedge1} it can be seen that the new lower bounds are often better than the ones from the literature.

However, only in two cases mixed metric dimension equals lower bound (1 from literature and 1 from the new ones). In 4 cases lower bounds have difference 1 from exact values.
It should be noted that all 7 lower bounds should be used in union since different lower bounds are applicable for different graphs and no one is uniquely dominant over the others. The important feature of presented lower bounds is that their calculation complexity is much smaller in comparison with standard/edge/mixed metric dimension problem complexity.

\begin{table}
\small
\caption{Graph characteristic}
\label{jedge}
\begin{tabular}{|l|l|l|l|l|l|l|}

\hline
Num&  Name & $|V|$& $|E|$& $\beta(G)$&$\beta_E(G)$& Another  notions \\
\hline
1.& Rook's graph & 36& 180& 7 & 8 & srg(36,10,4,2) \\
\hline

2.& 9-triangular graph & 36& 252& 6 & 32 & Johnson graph; srg(36,14,7,4) \\

\hline
3.& Clebsch graph  &16 &40 & 4 &9& srg(16,5,0,2)\\
\hline

4.& Generalized quadrangle  &27 &135& 5 & 18 & srg(27,10,1,5)\\
\hline

5.& Hypercube  $Q_5$ &32 &80& 4 & 4 & $5-$cube graph \\

\hline

6.& Kneser  (7,2)&21 &105&  5 & 12 & srg(21,10,3,6)\\

 \hline
7.& Mobius Kantor  &16 &24& 4 & 4 & Generalized Petersen $GP(8,3)$ \\

 \hline
8.& Paley graph  &13 &39& 4 & 6 & srg(13,6,2,3)\\
 \hline
9.& Petersen graph  &10 &15& 3 & 4 & Generalized Petersen $GP(5,2)$\\
 \hline
10.& Small graph 6 vert. &6 &11& 3 & 4 & \\
 \hline
11.& Hamming H(2,6)& 36 & 180 &  7 & 8 & $K_6 \Box K_6$\\
 \hline
12.& Hamming H(3,3) & 27 & 81 &  4 & 5 & $K_3 \Box K_3  \Box K_3$\\
 \hline
\end{tabular}
\end{table}

\begin{table}
\small
\caption{Direct comparison of lower bounds for some  graphs}
\label{jedge1}
\begin{tabular}{|l|l|l|l|l|l|l|l|l|l|}
\hline

&\multicolumn{8}{|c|}{LB from lit.}\\
 \hline
Num& L1 &L2 & L3&L4&N1&N2&N3& $\beta_{M}(G)$\\
\hline
\hline

1.   & 4& 5 &0 & 6 & 5& 6 &8&9 \\
 \hline
 2.&4& 5 & 0& 18 & 5 & 9 & 8 & 32\\
 \hline
3. &   3 & 4 & 0& 4 & 4 & 5 & 5& 9\\
 \hline
  4.  & 4 & 5 &0 & 4 & 5 & 6 & 8& 18\\
 \hline
 5.&   3& \bf{\underline{4}}&0 & 2& $\underline{\textbf{4}}$   & 2 & 3& 4 \\
 \hline
 6.&   4& 5 &0& 4& 5 & 6 & 6& 12\\
 \hline
 7.&   2& 3 &0& 2& 3 & 3 & 3& 4\\
 \hline
8. &  3& 4 &0& 4 & 4 & 5 & 5& 6\\
 \hline
  9.   & 2 & 3 &0&  4& 3 & 4 & 4&6 \\
 \hline
 10.&    2 & 2 &$\underline{\textbf{5}}$ &$\underline{\textbf{5}}$  & 3 & 4 & 3& 5\\
 \hline
 11.  & 4 & 5 &0&6 &5 & 6  & 8 & 9\\
 \hline
 12.  &  3& 4 &0& 3& 4& 3  & 4 &6 \\
 \hline
\end{tabular}
\end{table}

Using the results shown in the Table $\ref{jedge1}$, it can be concluded that these bounds do not give some perfect results, as expected, but we can see that there are graphs on which lower bounds L1 and L2 have been reached, as well as new lower bounds N1 and N2.

Finally, it must be noted that even with additional comparison, the representative sample is not statistically quite correct, but is given to illustrate usage and efficiency of presented lower bounds in some capacity. Graphs that we took in consideration varied from order 10 up to order 36. Second number (order 36) is chosen because exact value of their mixed metric dimensions could be quickly determined. Since problem of finding exact mixed metric dimension is NP-hard, our selection do not include large graphs.

\section{Conclusions}

In this paper problem of mixed metric dimension was considered. Since the problem of finding exact values of mixed metric dimension is NP-hard in general case, determining quality lower bounds is of major interest. Paper presented three new lower bounds in accordance with minimization nature of the problem. Next,
one of these lower bounds is tight for torus graph, so it is used for obtaining exact value of its mixed metric dimension.

In the sequel, quality of new lower bounds was compared with four known lower bounds from the literature. Testing was performed on two groups: all 21 connected graphs of 5 vertices and 12 well-known graphs with 10 up to 36 vertices. For the first group, one of proposed lower bounds in 20 out of 21 cases reached value of mixed metric dimension. For the second group, situation is quite opposite, so only in 2 cases proposed lower bound reached value of mixed metric dimension.

Further research could be developed  in two main directions. The first direction  could be finding
general upper bounds for mixed metric dimension. The second direction is usage of presented lower bounds in order to determine the exact values of mixed metric dimension of some well-known classes of graphs.



\section*{References}

\begin{thebibliography}{10}
\expandafter\ifx\csname url\endcsname\relax
  \def\url#1{\texttt{#1}}\fi
\expandafter\ifx\csname urlprefix\endcsname\relax\def\urlprefix{URL }\fi
\expandafter\ifx\csname href\endcsname\relax
  \def\href#1#2{#2} \def\path#1{#1}\fi

\bibitem{sla75}
P.~J. Slater, Leaves of trees, Congr. Numer 14~(549-559) (1975) 37.

\bibitem{har76}
F.~Harary, R.~Melter, On the metric dimension of a graph, Ars Combin
  2~(191-195) (1976) 1.

\bibitem{khul96}
S.~Khuller, B.~Raghavachari, A.~Rosenfeld, Landmarks in graphs, Discrete
  Applied Mathematics 70~(3) (1996) 217--229.

\bibitem{john93}
M.~Johnson, Structure-activity maps for visualizing the graph variables arising
  in drug design, Journal of Biopharmaceutical Statistics 3~(2) (1993)
  203--236.

\bibitem{char00}
G.~Chartrand, L.~Eroh, M.~A. Johnson, O.~R. Oellermann, Resolvability in graphs
  and the metric dimension of a graph, Discrete Applied Mathematics 105~(1-3)
  (2000) 99--113.

\bibitem{mel84}
R.~A. Melter, I.~Tomescu, Metric bases in digital geometry, Computer Vision,
  Graphics, and Image Processing 25~(1) (1984) 113--121.

\bibitem{bri03}
R.~C. Brigham, G.~Chartrand, R.~D. Dutton, P.~Zhang, Resolving domination in
  graphs, Mathematica Bohemica 128~(1) (2003) 25--36.

\bibitem{char03}
G.~Chartrand, V.~Saenpholphat, P.~Zhang, The independent resolving number of a
  graph, Mathematica Bohemica 128~(4) (2003) 379--393.

\bibitem{seb04}
A.~Seb{\H{o}}, E.~Tannier, On metric generators of graphs, Mathematics of
  Operations Research 29~(2) (2004) 383--393.

\bibitem{oka10}
F.~Okamoto, B.~Phinezy, P.~Zhang, The local metric dimension of a graph,
  Mathematica Bohemica 135~(3) (2010) 239--255.

\bibitem{kel18}
A.~Kelenc, N.~Tratnik, I.~G. Yero, Uniquely identifying the edges of a graph:
  the edge metric dimension, Discrete Applied Mathematics 251 (2018) 204--220.

\bibitem{yer17}
A.~Kelenc, D.~Kuziak, A.~Taranenko, I.~G. Yero, Mixed metric dimension of
  graphs, Applied Mathematics and Computation 314 (2017) 429--438.

\bibitem{raz19}
H.~Raza, J.-B. Liu, S.~Qu, On mixed metric dimension of rotationally symmetric
  graphs, IEEE Access.

\bibitem{mil20}
M.~Milivojevi\'c~Danas, Mixed metric dimension of flower snarks and wheels,
  arXiv preprint.

\bibitem{fil19}
V.~Filipovi{\'c}, A.~Kartelj, J.~Kratica, Edge metric dimension of some
  generalized {P}etersen graphs, Results in Mathematics 74~(4) (2019) 182.

\end{thebibliography}
\end{document}